\documentclass[12pt, a4paper]{amsart}
\usepackage{amsfonts}
\usepackage{tikz}
\usepackage{amscd,amssymb}
\usepackage{amstext,amsmath,amssymb,amsfonts}
\usepackage{times}
\usepackage[T1]{fontenc}
\usepackage{ajmaa}
\usepackage[
    pdftex,
    pdfpagemode=UseOutlines,
    bookmarksopen=true,
    baseurl={https://ajmaa.org/},
    pdfstartview={Fit}
]{hyperref}

\theoremstyle{plain}

\numberwithin{equation}{section}

\issueinfo{x}{x}{x}{202x}
\dateinfo{ 20 April 2020 }{}{}
\nmpages{3}

\begin{document}
\title[RESIDUAL-BASED A POSTERIORI ERROR
ESTIMATES]{A POSTERIORI ERROR
	ANALYSIS FOR A LAGRANGE MULTIPLIER METHOD FOR A STOKES/BIOT FLUID-POROELASTIC STRUCTURE INTERACTION MODEL}

\author[HOUEDANOU K. W.]{HOUEDANOU Koffi Wilfrid}
	\address{D\'epartement de Math\'ematiques/Facult\'e des Sciences et Techniques (FAST)/Universit\'e d'Abomey-Calavi (UAC)}
\email{\href{mailto: Author <khouedanou@yahoo.fr>}{khouedanou@yahoo.fr}}




\begin{abstract}
	In this work we develop an a posteriori error analysis of a conforming mixed finite element method for 
	solving the coupled problem arising in the interaction between a free fluid and a fluid in a poroelastic medium on 
isotropic meshes in $\mathbb{R}^d$, $d\in\{2,3\}$. The approach utilizes the semi-discrete formulation  proposed by  Ilona Ambartsumyan et al.  in \cite{1}.
The a posteriori error estimate is based on a suitable evaluation on the residual of the finite element
solution. It is proven that the a posteriori error estimate provided in this paper is both reliable and efficient.
The proof of reliability makes use of suitable auxiliary problems, diverse continuous inf-sup conditions satisfied by the bilinear forms involved, Helmholtz decomposition, and local approximation properties
of the Cl\'ement interpolant.
On the other hand, inverse inequalities, and the localization technique based on simplexe-bubble and face-bubble
functions are the main tools for proving the efficiency of the estimator.
Up to minor modifications,
our analysis can be extended to other finite element subspaces yielding a stable Galerkin scheme.
\end{abstract}
\maketitle
\newpage 
\section{Introduction\label{s1}}
In this paper, we develop an a posteriori error analysis for solving the interaction of a free incompressible viscous Newtonian fluid with a fluid within a poroelastic medium. This is a challenging multiphysics problem with applications to predicting and controlling processes arising in groundwater flow in fractured aquifers, oil and gas extraction, arterial flows, and industrial filters. 
In these applications, it is important to model properly the interaction
between the free fluid with the fluid within the porous medium, and to take into account the effect of
the deformation of the medium. For example, geomechanical effects play an important role in hydraulic
fracturing, as well as in modeling phenomena such as subsidence and compaction.

We adopt the Stokes equations to model the free fluid and the Biot system \cite{2} for the fluid in the
poroelastic media. In the latter, the volumetric deformation of the elastic porous matrix is complemented
with the Darcy equation that describes the average velocity of the fluid in the pores. The model features two different kinds of coupling across the interface: Stokes-Darcy coupling \cite{AHN:15,SGR:2016,CGOS:2015,46,47,HJA:2017,HA:2016, H:2019,H:2015} and fluid-structure interaction (FSI) 
\cite{3',4',5',6',7'}.

The well-posedness of the mathematical model based on the Stokes-Biot system for the coupling between a fluid and a poroelastic structure is studied in \cite{8'}. A numerical study of the problem, using a Navier-Stokes equations for the fluid, is presented in  \cite{3',9'}, utilizing a variational multiscale approach to stabilized the finite element spaces. The problem is solved using both a monolithic and a partitioned approach, with the latter requiring subiterations between the two problems.

A posteriori error estimators are computable quantities, expressed in terms of the discrete solution and of the data that measure the actual discrete errors without the knowledge of the exact solution. They are essential to design adaptive mesh refinement algorithms which aqui-distribute the computational effort and optimize the approximation efficiency. Since the pioneering work of Babu\v{s}ka and Rheinboldt 
\cite{1978',1993,BR78a,BR78b}, adaptive finite element methods based on a posteriori error estimates have been extensively investigated.

In \cite{1}, 
semidiscrete continuous-in-time approximation
 has been proposed for the weak coupled mixed formulation.
For the discretization of the fluid velocity and pressure the authors have used the finite elements  which include the MINI-elements, the Taylor-Hood elements and the conforming Crouzeix-Raviart elements. For the discretization of the porous medium problem they choose the spaces that include Raviart-Thomas and Douglas-Marini elements. An a priori error analysis is performed with some numerical tests confirming the convergence rates. To our best knowledge, there is no a posteriori error estimation for the Stokes/Biot fluid-poroelastic structure interaction model for finite element methods. Here we develop such a posteriori error analysis for the semi-discrete conforming finite element methods.
We have got a new family of a local indicator error $\Theta_K$ (see Definition \ref{dindK},
eq. (\ref{indK})) and global $\Theta$ (eq. \ref{indKg}). We prove that our indicators error are efficiency and reliability, and then, are
optimal.
The global inf-sup condition is the main tool yielding the reliability. In turn, The local
efficiency result is derived using the technique of bubble function introduced by R. Verf\"{u}rth
 \cite{verfurth:96b} and used in similar context by C. Carstensen \cite{carstensenandall}.
 
 The paper is organized as follows. Some preliminaries and notation are given in Section \ref{s2}. In
 Section \ref{s3'}, the a posteriori error estimates are derived. We offer our conclusion and the further
 works in Section \ref{s6}.
\section{Preliminaries and notations}\label{s2}
\subsection{Stokes-Biot model problem}
We consider a multiphysics model problem for free fluid's interaction with a flow in a deformable porous media, where the simulation domain $\Omega\subset \mathbb{R}^d$, $d=2,3$, is a union of non-overlapping regions $\Omega_f$ and $\Omega_p$. Here $\Omega_f$ is a free fluid region with flow governed by the Stokes equations and $\Omega_p$ is a poroelastic material governed by the Biot system. For simplicity of notation, we assume that each region is connected. The extension to non-connected regions is straightforward. Let $\Gamma_{fp}=\partial \Omega_f\cap \partial \Omega_p$ (see Fig. \ref{F1}).
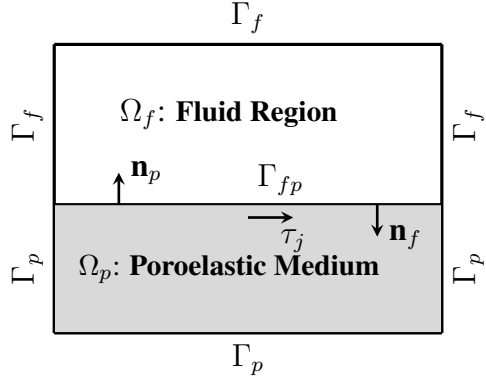
\begin{figure}[http!!!]
	\centering
	\begin{center}
		\tikzstyle{grisEncadre}=[thick, dashed, fill=gray!20]
		\begin{tikzpicture}[scale=0.85]
		color=gray!100;
		\draw [very thick](1,1)--(7,1);
		\draw [very thick](1,1)--(1,5.5);
		\draw [very thick](1,5.5)--(7,5.5);
		\draw [very thick](7,1)--(7,5.5);
		\draw [very thick](1,3)--(7,3);
		\draw [black,fill=gray!30] (1,1) -- (7,1) -- (7,3) --(1,3) -- cycle;
		\draw (3.7,1.6) node [above]{$\mbox{ \small\small $\Omega_p$:  \textbf{\small \small Poroelastic Medium} }$};
		\draw (3.7,4) node [above]{$\mbox{  $\Omega_f$:  \textbf{\small \small Fluid Region} }$};
		\draw [>=stealth,->] [line width=1pt](2,3)--(2,3.5) node [right]{$\textbf{n}_p$};
		\draw [>=stealth,->] [line width=1pt](6,3)--(6,2.5) node [right]{$\textbf{n}_f$};
		\draw [>=stealth,->] [line width=1pt](4,2.8)--(4.7,2.8) node [below]{$\tau_j$};
		\draw  (4.5,3.8) node [below]{$\Gamma_{fp}$};
		\draw[line width=0.5pt](1,1)--(1,3) node[midway,above,sloped]{$\Gamma_p$};
		\draw[line width=0.5pt](1,1)--(7,1) node[midway,below,sloped]{$\Gamma_p$};
		\draw[line width=0.5pt](7,1)--(7,3) node[midway,below,sloped]{$\Gamma_p$};
		\draw[line width=0.5pt](1,3)--(1,5.5) node[midway,above,sloped]{$\Gamma_f$};
		\draw[line width=0.5pt](1,5.5)--(7,5.5) node[midway,above,sloped]{$\Gamma_f$};
		\draw[line width=0.5pt](7,3)--(7,5.5) node[midway,below,sloped]{$\Gamma_f$};
		\end{tikzpicture}
	\end{center}
	\caption{\footnotesize{Global domain $\Omega$ consisting of the fluid region $\Omega_f$ and the 
			poroelastic media region $\Omega_p$ separated by the interface $\Gamma_{fp}$.}}
	\label{F1}
\end{figure}
\mbox{  }\\
Let $(\textbf{u}_{\star},p_{\star})$ be the velocity-pressure pair in $\Omega_{\star}$, $\star=f, p$, and let $\eta_p$ be the displacement in $\Omega_p$. Let $\mu> 0$ be the fluid viscosity, let $\textbf{f}_{\star}$ be the body force terms, and let $q_{\star}$ be external source or sink terms. Let $\textbf{D}(\textbf{u}_f)$ and $\sigma_f(\textbf{u}_f,p_f)$ denote, respectively, the deformation rate tensor and the stress tensor:
$$
\textbf{D}(\textbf{u}_f)=\frac{1}{2}\left(\nabla \textbf{u}_f+{\nabla \textbf{u}_f}^T\right), \mbox{     and   }  
\sigma_f(\textbf{u}_f,p_f)=-p_f\textbf{I}+2\mu\textbf{D}(\textbf{u}_f).
$$
In the free fluid region $\Omega_f$, $(\textbf{u}_f,p_f)$ satisfy the Stokes equations:
\begin{eqnarray}\label{stokes1}
-\nabla\cdot\sigma_f(\textbf{u}_f,p_f)=\textbf{f}_f \mbox{      in   } \Omega_f\times (0,T]\\\label{stokes2}
\nabla\cdot\textbf{u}_f=q_f \mbox{    in   } \Omega_f\times (0,T]
\end{eqnarray}
where $T> 0$ is the final time. Let $\sigma_e(\eta_p)$ and $\sigma_p(\eta_p,p_p)$ be the elastic and poroelastic stress tensors, respectively: 
\begin{eqnarray}\label{poroelastic}
	\sigma_e(\eta_p)=\lambda_p\left(\nabla\cdot\eta_p\right)\textbf{I}+2\mu_p\textbf{D}(\eta_p), \mbox{       }  \sigma_p(\eta_p,p_p)=\sigma_e(\eta_p)-\alpha p_p \textbf{I},
\end{eqnarray}
where $0< \lambda_{\min}\leq \lambda_p(\textbf{x})\leq \lambda_{\max}$ and  $0< \mu_{\min}\leq \mu_p(\textbf{x})\leq \mu_{\max}$ are the Lam\'e parameters and $0\leq \alpha\leq 1$ is the Biot-Wallis constant. The poroelasticity region $\Omega_p$ is governed by the quasi-static Biot system \cite{1}:
\begin{eqnarray}\label{darcy1}
	-\nabla\cdot \sigma_p(\eta_p,p_p)=\textbf{f}_p, \mbox{  } \mu K^{-1}\textbf{u}_p+\nabla p_p=0 \mbox{  in  } \Omega_p\times (0,T],\\\label{darcy2}
	\frac{\partial}{\partial t}\left(s_0p_p+\alpha\nabla\cdot \eta_p\right)+\nabla\cdot \textbf{u}_p=q_p \mbox{
	 in  } \Omega_p\times (0,T],
\end{eqnarray}
where $s_0\geq 0$ is a storage coefficient and $K$ the symmetric and uniformly positive definite rock permeability tensor, satisfying, for some constants $0< k_{\min}\leq k_{\max}$,
$$
\forall \xi \in\mathbb{R}^d, k_{\min}\xi^T\xi\leq \xi^TK(\textbf{x})\xi\leq k_{\max} \xi^T\xi, \forall \textbf{x}\in\Omega_p.
$$
Following \cite{2}, the interface conditions on the fluid-poroelasticity interface $\Gamma_{fp}$ are mass conservation, balance of stresses, and the Beavers-Joseph-Saffman (BJS) condition \cite{23} modeling slip with friction:
\begin{eqnarray}\label{i1}
	\textbf{u}_f\cdot\textbf{n}_f+\left(\frac{\partial \eta_p}{\partial t}+\textbf{u}_p\right)\cdot\textbf{n}_p=0  \mbox{    on    } \Gamma_{fp}\times (0,T],\\\label{i2}
	-(\sigma_f\textbf{n}_f)\cdot \textbf{n}_f=p_p,  \mbox{  }  \sigma_f\textbf{n}_f+\sigma_p\textbf{n}_p=0 \mbox{       on    }  \Gamma_{fp}\times (0,T],\\\label{i3}
	-(\sigma_f\textbf{n}_f)\cdot\tau_{f,j}=\mu\alpha_{BJF} \sqrt{K_j^{-1}}\left(\textbf{u}_f-\frac{\partial \eta_p}{\partial t}\right)\cdot\tau_{f,j} \mbox{     on   } \Gamma_{fp}\times (0,T],
\end{eqnarray}
where $\textbf{n}_f$ and $\textbf{n}_p$ are the outward unit normal vectors to $\partial \Omega_f$, and $\partial \Omega_p$, respectively, $\tau_{f,j}$, $1\leq j\leq d-1$, is an orthogonal system of unit tangent vectors on $\Gamma_{fp}$, $K_j=\left(K\tau_{f,j}\right)\cdot\tau_{f,j}$, and $\alpha_{BJS}\geq 0$ is an experimentally determined friction coefficient. We note that continuity of flux constraints the normal velocity of the solid skeleton, while the $BJS$ condition accounts for its tangential velocity.

The above system of equations needs to be complemented by a set of boundary and initial conditions. Let $\Gamma_f=\partial \Omega_f\cap \partial \Omega$ and $\Gamma_p=\partial \Omega_p\cap \partial\Omega$. Let $\Gamma_p=\Gamma_p^D\cup\Gamma_{p}^N$. We assume for simplicity homogeneous boundary conditions:
$$ \textbf{u}_f=0 \mbox{   on  } \Gamma_f\times (0,T], \mbox{  } \eta_p=\textbf{0} \mbox{
 on  } \Gamma_p\times (0,T], \mbox{   } p_p=0 \mbox{   on  } \Gamma_p^D\times (0,T], \mbox{    } \textbf{u}_p\cdot \textbf{n}_p=0 \mbox{   on  }   \Gamma_p^N\times (0,T].$$
To ovoid the issue with restricting the mean value of the pressure, we assume that $|\Gamma_{p}^D|> 0$. We also assume that $\Gamma_p^D$ is not adjacent to the interface $\Gamma_{fp}$, i.e., $\mbox{dist}(\Gamma_p^D,\Gamma_{fp})\geq s> 0$. Non-homogeneous displacement and velocity conditions can be handled in a standard way by adding suitable extensions of the boundary data. The pressure boundary condition is natural in the mixed Darcy formulation, so non-homogeneous pressure data would lead to an additional boundary term. We further sat the initial conditions:
$$p_p(\textbf{x},0)=p_{p,0}(\textbf{x}), \mbox{     } \eta_p(\textbf{x},0)=\eta_{p,0}(\textbf{x}) \mbox{     in    } \Omega_p.$$
\subsection{Weak formulation}
In this part, we first introduce some Sobolev spaces \cite{Adams:1978} and norms.
If $W$ is a bounded domain of $\mathbb{R}^d$ and $m$ 
is a non negative integer, the Sobolev space $H^m(W)=W^{m,2}(W)$ is 
defined in the usual way with the usual norm $\parallel\cdot\parallel_{m,W}$ and semi-norm $|.|_{m,W}$. In particular, 
$H^0(W)=L^2(W)$ and we write $\parallel\cdot\parallel_W$ for $\parallel\cdot\parallel_{0,W}$.
Similarly we   denote by
$(\cdot,\cdot)_{W}$  the $L^2(W)$ $[L^2(W)]^d$ or $ [L^2(W)]^{d\times d}$ inner product.
For shortness if $W$ is equal to $\Omega$, we will drop  the index $\Omega$, while  for any $m\geq 0$, 
$\parallel\cdot\parallel_{m,\star}=\parallel\cdot\parallel_{m,\Omega_{\star}}$, $|.|_{m,\star}=|.|_{m,\Omega_{\star}}$ 
and $(.,.)_{\star}=(\cdot,\cdot)_{\Omega_{\star}}$, for $\star=f,p$.
The space  $H_0^m(\Omega)$ denotes the closure of $C_0^{\infty}(\Omega)$ in $H^{m}(\Omega)$. Let $[H^m(\Omega)]^d$ be the space of
vector valued functions $\textbf{v}=(v_1,\ldots,v_d)$ with components  $v_i$ in $H^m(\Omega)$. The 
norm and the seminorm on $[H^m(\Omega)]^d$ are given by 
\begin{eqnarray}
\parallel\textbf{v}\parallel_{m,\Omega}:=\left(\sum_{i=0}^d\parallel v_i\parallel_{m,\Omega }^2\right)^{1/2} 
\mbox{  and  } |\textbf{v}|_{m,\Omega}:= \left(\sum_{i=0}^d |v_i|_{m,\Omega}^2\right)^{1/2}.
\end{eqnarray}
For a connected open subset of the boundary $E\subset \partial \Omega_f\cup\partial \Omega_p$, we write 
$\langle.,.\rangle_{E}$ for the $L^2(E)$ inner product (or duality pairing), that is, for scalar valued functions $\lambda$, 
$\sigma$ one defines:
\begin{eqnarray}
\langle\lambda,\sigma\rangle_{E}:=\int_{E} \lambda\sigma ds
\end{eqnarray}
In the following we derive a Lagrange multiplier type weak formulation of the system, which will be the basis for our finite element approximation.
Let 
\begin{eqnarray}
	\textbf{V}_f=\left\{\textbf{v}_f\in H^1(\Omega_f)^d: \textbf{v}_f=\textbf{0} \mbox{  on  } \Gamma_f\right\}, \mbox{     }  W_f=L^2(\Omega_f),\\
	\textbf{V}_p=\left\{ \textbf{v}_p\in H(\dive; \Omega_p): \textbf{v}_p\cdot\textbf{n}_p=0 \mbox{   on   } \Gamma_p^N\right\}, \mbox{        }    W_p=L^2(\Omega_p),\\ 
	\textbf{X}_p= \left\{\xi_p\in H^1(\Omega_p)^d: \xi_p=\textbf{0} \mbox{    on   } \Gamma_p\right\},         
\end{eqnarray}
where $H(\dive;\Omega_p)$ is the space of $L^2(\Omega_p)^d$-vectors with divergence in $L^2(\Omega_p)$ with a norm 
$$
\parallel\textbf{v}\parallel_{H(\dive;\Omega_p)}^2:=\parallel\textbf{v}\parallel_{\Omega_p}^2+
\parallel\nabla\cdot \textbf{v}\parallel_{\Omega_p}^2.
$$
We define the global velocity and pressure spaces as
\begin{eqnarray*}
	\textbf{V}=\left\{ \textbf{v}=(\textbf{v}_f,\textbf{v}_p)\in \textbf{V}_f\times \textbf{V}_p\right\}, \mbox{       } W=\left\{w=(w_f,w_p)\in W_f\times W_p\right\},
\end{eqnarray*}
with norms
$$
\parallel\textbf{v}\parallel_{\textbf{V}}^2:=\parallel\textbf{v}_f\parallel_{\Omega_f}^2+
\parallel \textbf{v}_p\parallel_{H(\dive;\Omega_p)}^2, \mbox{  and  } 
\parallel w\parallel_{W}^2=\parallel w_f\parallel_{\Omega_f}^2+\parallel w_p\parallel_{\Omega_p}^2.
$$
The weak formulation is obtained by mutiplying the equations in each region by suitable test functions, integrating by parts the second order terms in space, and utilizing the interface and boundary conditions.
\\Let define
\begin{eqnarray*}
	a_f(\textbf{u}_f,\textbf{v}_f)&:=&\left(2\mu\textbf{D}(\textbf{u}_f),\textbf{D}(\textbf{v}_f)\right)_{\Omega_f},\\
	a_p^d(\textbf{u}_p,\textbf{v}_p)&:=& \left(\mu K^{-1}\textbf{u}_p,\textbf{v}_p\right)_{\Omega_p},\\
	a_p^e(\eta_p,\xi_p)&:=&\left(2\mu \textbf{D}(\eta_p),\textbf{D}(\xi_p)\right)_{\Omega_p}+\left(\lambda_p\nabla\cdot\eta_p,\nabla\cdot \xi_p\right)_{\Omega_p}
\end{eqnarray*}
be the bilinear forms related to Stokes, Darcy and the elasticity operator, respectively. Let 
$$b_{\star} (\textbf{v},w):=-(\nabla\cdot \textbf{v},w)_{\Omega_{\star}}.$$
Integration by parts in (\ref{stokes1}) and the two equations in (\ref{darcy1}) lead to the interface term
$$I_{\Gamma_{fp}}=-\langle\sigma_f\textbf{n}_f,\textbf{v}_f\rangle_{\Gamma_{fp}}-\langle \sigma_p\textbf{n}_p,\xi_p\rangle_{\Gamma_{fp}}+\langle p_p,\textbf{v}_p\cdot\textbf{n}_p\rangle_{\Gamma_{fp}}.$$
Using the first condition for balance of normal stress in (\ref{i2}) we set
$$\lambda=-(\sigma_f\textbf{n}_f)\cdot\textbf{n}_f=p_p \mbox{     on   } \Gamma_{fp},$$
which will be used as a Lagrange multiplier to impose the mass conservation interface condition (\ref{i1}). Utilizing the BJS condition (\ref{i3}) and the second condition for balance of stresses in (\ref{i2}), we obtain 
$$I_{\Gamma_{fp}}=a_{BJS}(\textbf{u}_f,\partial_t \eta_p; \textbf{v}_f, \xi_p)+b_{\Gamma} (\textbf{v}_f,\textbf{v}_p,\xi_p,\lambda),$$
where
\begin{eqnarray*}
a_{BJS}(\textbf{u}_f,\eta_p;\textbf{v}_f,\xi_p)&=&\displaystyle\sum_{j=1}^{d-1}
\langle \mu\alpha_{BJS}\sqrt{K_j^{-1}}(\textbf{u}_f-\eta_p)\cdot\tau_{f,j},(\textbf{v}_f-\xi_p)\cdot\tau_{f,j}\rangle_{\Gamma_{fp}}\\
b_{\Gamma}(\textbf{v}_f,\textbf{v}_p,\xi_p;\mu)&=& \langle \textbf{v}_f\cdot\textbf{n}_f+(\xi_p+\textbf{v}_p)\cdot\textbf{n}_p,\mu\rangle_{\Gamma_{fp}}.
\end{eqnarray*}
For the well-posedness of $b_{\Gamma}$ we require that $\lambda\in\bigwedge=({\textbf{V}_p\cdot\textbf{n}_p}_{|\Gamma_{fp}})'$. According to the normal trace theorem, since $\textbf{v}_p\in\textbf{V}_p\subset H(\dive;\Omega_p)$, then $\textbf{v}_p\cdot\textbf{n}_p\in H^{-1/2}(\partial \Omega_p)$. Furthermore, since $\textbf{v}_p\cdot \textbf{n}_p=0$ on $\Gamma_{p}^N$ and $\mbox{ dis } (\Gamma_p^D,\Gamma_{fp})\geq s> 0$, then $\textbf{v}_p\cdot\textbf{n}_p\in H^{-1/2}(\Gamma_{fp})$, see, e.g. \cite{10'} Therefore we take $\bigwedge=H^{1/2}(\Gamma_{fp})$. 

The Lagrange multiplier variational formulation is: for $t\in (0,T]$, find $\textbf{u}_{f}(t)\in \textbf{V}_f$, $p_f(t)\in W_f$, $\textbf{u}_p(t)\in \textbf{V}_p$, $p_p(t)\in W_p$, $\eta_p(t)\in \textbf{X}_p$, and $\lambda(t)\in \bigwedge$, such that $p_p(0)=p_{p,0}$, $\eta_p(0)=\eta_{p,0}$, and for all $\textbf{v}_f\in\textbf{V}_f$, $w_f\in W_f$, $\textbf{v}_p\in \textbf{V}_p$, $w_p\in W_p$, $\xi_p\in \textbf{X}_p$, and $\mu\in\bigwedge$,
\begin{eqnarray}\label{f1}
a_f(\textbf{u}_f,\textbf{v}_f)+a_p^d(\textbf{u}_p,\textbf{v}_p)
+a_p^e(\eta_p,\xi_p)&+&a_{BJS}(\textbf{u}_f,\partial_t\eta_p;
\textbf{v}_f,\xi_p)\\\nonumber
+b_f(\textbf{v}_f,p_f)+b_p(\textbf{v}_p,p_p)
+ \alpha b_p(\xi_p,p_p)+b_{\Gamma} (\textbf{v}_f,\textbf{v}_p,\xi_p;\lambda)&=&(\textbf{f}_f,\textbf{v}_f)_{\Omega_f}+(\textbf{f}_p,\xi_p)
_{\Omega_p}\\\nonumber
(s_0\partial_t p_p,w_p)_{\Omega_p}-\alpha b_p(\partial_t\eta_p,w_p)
&-&b_p(\textbf{u}_p,w_p)-b_f(\textbf{u}_f,w_f)\\\label{f2}
&=& (q_f,w_f)_{\Omega_f}+(q_p,w_p)_{\Omega_p},\\\label{f3}
b_{\Gamma}(\textbf{u}_f,\textbf{u}_p,\partial_t\eta_p;\mu)&=&0.
\end{eqnarray}
Where we used the notation $\partial_t=\frac{\partial}{\partial t}$.

The assumptions on the fluid viscosity $\mu$ and the material coefficients $K$, $\lambda_p$, and $\mu_p$ imply that the bilinear forms $a_f(\cdot,\cdot)$, $a_p^d(\cdot,\cdot)$, and $a_p^e(\cdot,\cdot)$ are coercive and continuous in the appropriate norms. In particular, there exist positive constants $c^f$, $c^p$, $c^e$, $C^f$, $C^p$, $C^e$ such that:
\begin{eqnarray}\label{ex1}
	c^f\parallel \textbf{v}_f\parallel_{H^1(\Omega_f)}^2 &\leq& a_f(\textbf{v}_f,\textbf{v}_f),
	\forall \textbf{v}_f\in \textbf{V}_f,\\\label{ex2}
	a_f(\textbf{v}_f,\textbf{q}_f)&\leq& C^f\parallel\textbf{v}_f\parallel_{H^1(\Omega_f)}
	\parallel \textbf{q}_f\parallel_{H^1(\Omega_f)}, \mbox{  } \forall \textbf{v}_f, \textbf{q}_f\in\textbf{V}_f,\\\label{ex3}
	c^p\parallel\textbf{v}_p\parallel_{L^2(\Omega_p)}^2 &\leq &a_p^d(\textbf{v}_p,\textbf{v}_p),  \forall \textbf{v}_p\in \textbf{V}_p,\\\label{ex4}
	a_p^d(\textbf{v}_p,\textbf{q}_p)&\leq& C^p\parallel \textbf{v}_p\parallel_{L^2(\Omega_p)}
	\parallel \textbf{q}_p\parallel_{L^2(\Omega_p)}, \forall \textbf{v}_p, \textbf{q}_p\in \textbf{V}_p,
	\\\label{ex5}
	c^e\parallel\xi_p\parallel_{H^1(\Omega_p)}^2 &\leq& a_p^e(\xi_p,\xi_p)\forall \xi_p\in\textbf{X}_p\\
	\label{ex6}
	a_p^e(\xi_p,\zeta_p)&\leq& C^e\parallel \xi_p\parallel_{H^1(\Omega_p)}\parallel\zeta_p\parallel_{H^1(\Omega_p)} \forall \xi_p,\zeta_p\in\textbf{X}_p,
\end{eqnarray}
where (\ref{ex1})-(\ref{ex2}) and (\ref{ex5})-(\ref{ex6}) hold true thanks to Poincar\'e inequality and (\ref{ex5})-(\ref{ex6}) also relies on Korn's inequality.\\
In summary, from \cite[Corollary 3.1 , Page 7]{1}, the following result holds:
\begin{theorem}
	There exists a unique solution $(\textbf{u}_f,p_f,\textbf{u}_p,p_p,\eta_p,\lambda)\in 
	L^{\infty}(0,T; \textbf{V}_f)\times L^{\infty}(0,T;W_f)\\\times L^{\infty}(0,T;\textbf{V}_p)\times 
	W^{1,\infty}(0,T;W_p)\times W^{1,\infty}(0,T;\textbf{X}_p)\times L^{\infty}(0,T;\Lambda)$ to the problem (\ref{f1})-(\ref{f3}).
\end{theorem}
\subsection{Semi-discrete formulation}
Let $\cT_h^f$ and $\cT_h^p$ be shape-regular and quasi-uniform partition of $\Omega_f$ and $\Omega_p$, respectively, both consisting of affine elements with maximal element diameter $h$. The two partitions may be non-matching at the interface $\Gamma_{fp}$. For the discretization of the fluid velocity and pressure we choose finite element spaces $\textbf{V}_{f,h}\subset \textbf{V}_f$ and $W_{f,h}\subset W_f$, which are assumed to be inf-sup stable. Examples of such spaces include the MINI elements, the Taylor-Hood elements and the conforming Crouzeix-Raviart elements. For the discretization of the porous medium problem we choose $\textbf{V}_{p,h}\subset \textbf{V}_p$ and $W_{p,h}\subset W_p$ to be any of well-known inf-sup stable mixed finite element spaces, such as the Raviart-Thomas or the  Brezzi-Douglas-Marini spaces. The global spaces are:
\begin{eqnarray*}
	\textbf{V}_h:=\left\{ \textbf{v}_h=(\textbf{v}_{f,h},\textbf{v}_{p,h})\in \textbf{V}_{f,h}\times \textbf{V}_{p,h}\right\}, \mbox{        }
	W_h=\left\{ w_h=(w_{f,h},w_{p,h})\in W_{f,h}\times W_{p,h}\right\}.
\end{eqnarray*}
We employ a conforming Lagrangian finite element space $\textbf{X}_{p,h}\subset \textbf{X}_p$ to approximate the structure displacement. Note that the finite element spaces $\textbf{V}_{f,h}$, $\textbf{V}_{p,h}$ and $\textbf{X}_{p,h}$ satisfy the prescribed homogeneous boundary conditions on the external boundaries. For the discrete Lagrange multiplier space we take 
$$
\Lambda_h={\textbf{V}_{p,h}\cdot \textbf{n}_p}_{|\Gamma_{fp}}.
$$
The semi-discrete continuous-in-time problem reads: given $p_{p,h}(0)$ and $\eta_{p,h}(0)$, for $t\in(0,T]$, find $\textbf{u}_{f,h}(t)\in\textbf{V}_{f,h}$, $p_{f,h}(t)\in W_{f,h}$, 
$\textbf{u}_{p,h}(t)\in \textbf{V}_{p,h}$, $p_{p,h}(t)\in W_{p,h}$, $\eta_{p,h}(t)\in \textbf{X}_{p,h}$, and $\lambda_h(t)\in\Lambda_h$ such that for all $\textbf{v}_{f,h}\in \textbf{V}_{f,h}$, 
$w_{f,h}\in W_{f,h}$, $\textbf{v}_{p,h}\in\textbf{V}_{p,h}$, $w_{p,h}\in W_{p,h}$, $\xi_{p,h}\in \textbf{X}_{p,h}$, and $\mu_h\in\Lambda_h$, 
\begin{eqnarray}\label{fh1}\nonumber
a_f(\textbf{u}_{f,h},\textbf{v}_{f,h})+a_p^d(\textbf{u}_{p,h},\textbf{v}_{p,h})
+a_p^e(\eta_{p,h},\xi_{p,h})\\
+a_{BJS}(\textbf{u}_{f,h},\partial_t\eta_{p,h};
\textbf{v}_{f,h},\xi_{p,h})
+b_f(\textbf{v}_{f,h},p_{f,h})\\\nonumber
+b_p(\textbf{v}_{p,h},p_{p,h})
+ \alpha b_p(\xi_{p,h},p_{p,h})+b_{\Gamma} (\textbf{v}_{f,h},\textbf{v}_{p,h},\xi_{p,h};\lambda_h)&=&(\textbf{f}_{f,h},\textbf{v}_{f,h})_
{\Omega_f}+(\textbf{f}_{p,h},\xi_{p,h})
_{\Omega_p}\\\nonumber
(s_0\partial_t p_{p,h},w_{p,h})_{\Omega_p}-\alpha b_p(\partial_t\eta_{p,h},w_{p,h})
&-&b_p(\textbf{u}_{p,h},w_{p,h})-b_f(\textbf{u}_{f,h},w_{f,h})\\\label{fh2}
&=& (q_f,w_{f,h})_{\Omega_f}+(q_p,w_{p,h})_{\Omega_p},\\\label{fh3}
b_{\Gamma}(\textbf{u}_{f,h},\textbf{u}_{p,h},\partial_t\eta_{p,h};\mu_h)&=&0.
\end{eqnarray}
We will take $p_{p,h}(0)$ and $\eta_{p,h}(0)$ to be suitable projections of the initial data $p_{p,0}$ and 
$\eta_{p,0}$. 

We introduce the errors for all variables:
$$\textbf{e}_f:=\textbf{u}_{f}-\textbf{u}_{f,h}, \mbox{
 }  \textbf{e}_p:=\textbf{u}_p-\textbf{u}_{p,h}, \mbox{     } \textbf{e}_s:=\eta_p-\eta_{p,h},\mbox{   }
e_{fp}:=p_f-p_{f,h}, \mbox{    } e_{pp}:=p_p-p_{p,h} \mbox{    and    } e_{\lambda}:=\lambda-\lambda_h.$$
The following results hold cf. \cite{1}:
\begin{theorem}
	[\textbf{A-priori error estimation} ]
	There exists a unique solution $(\textbf{u}_{f,h}, p_{f,h}, \textbf{u}_{p,h},p_{p,h},\eta_{p,h},\lambda_h)$ in 
	$L^{\infty}(0,T;\textbf{V}_{f,h})\times L^{\infty}(0,T;W_{f,h})\times L^{\infty}(0,T;\textbf{V}_{p,h})\times W^{1,\infty}(0,T;W_{p,h})\times 
	W^{1,\infty}(0,T;\textbf{X}_{p,h})\times  L^{\infty}(0,T;\Lambda_h)$ of the weak formulation 
	(\ref{fh1})-(\ref{fh3}) and if the solution 
	$(\textbf{u}_{f}, p_{f}, \textbf{u}_{p},p_p,\eta_{p},\lambda)\in\textbf{V}_f\times W_f\times \textbf{V}_p\times W_p\times \textbf{X}_p\times \Lambda$ of the continuous problem (\ref{f1})-(\ref{f3}) is smooth enough, then we have:
	\begin{eqnarray}\nonumber
		\parallel \textbf{e}_f\parallel_{L^2(0,T;H^1(\Omega_f))}+
		\parallel \textbf{e}_p\parallel_{L^2(0,T;L^2(\Omega_p))}\\
		+\parallel\textbf{e}_s \parallel_{L^{\infty}(0,T;H^1(\Omega_p))}
		+\parallel e_p\parallel_{L^{\infty}(0,T;L^2(\Omega_p))}\\\nonumber
		+\parallel e_f\parallel_{L^2(0,T;L^2(\Omega_f))}
		+\parallel e_{\lambda}\parallel_{L^2(0,T;\Lambda_h)}
		\leq C(\textbf{u}_{f}, p_{f}, \textbf{u}_{p},p_{p},\eta_{p},\lambda)h^{r}, r\geq 1.
	\end{eqnarray}

\end{theorem}

	For $\textbf{W}_h=(\textbf{v}_{f,h},w_{f,h},\textbf{v}_{p,h},w_{p,h},\xi_{p,h},\lambda_h)\in 
	L^{\infty}(0,T;\textbf{V}_{f,h})\times L^{\infty}(0,T;W_{f,h})\times L^{\infty}(0,T;\textbf{V}_{p,h})\times W^{1,\infty}(0,T;W_{p,h})\times 
	W^{1,\infty}(0,T;\textbf{X}_{p,h})\times  L^{\infty}(0,T;\Lambda_h)$,
we can subtract (\ref{fh1})-(\ref{fh3}) to (\ref{f1})-(\ref{f3}) to obtain the Galerkin orthogonality relation for all $t\in]0,T]$:
\begin{eqnarray*}
	a_f(\textbf{e}_f(t),\textbf{v}_{f,h}(t))+a_p^d(\textbf{e}_s(t),\xi_{p,h}(t))+a_{BJS}(\textbf{e}_f(t),
	\partial_t\textbf{e}_s(t);\textbf{v}_{f,h}(t),\xi_{p,h}(t))\\
	+b_f(\textbf{v}_{f,h}(t),e_f(t))+
	b_p(\textbf{v}_{p,h}(t),e_p(t))
	+\alpha b_p(\xi_p(t),e_p(t))\\+b_{\Gamma}(\textbf{v}_f(t),\textbf{v}_p(t),\xi_p(t),e_{\lambda}(t))+
	(s_0\partial_t e_p(t),w_{p,h}(t))\\
	-\alpha b_p(\partial_t(t) \textbf{e}_s(t),w_{p,h}(t))-b_p(\textbf{e}_p(t),w_{p,h}(t))-b_f(\textbf{e}_f(t),w_{f,h}(t))=0.
\end{eqnarray*}
\section{A-posteriori error analysis}\label{s3'}
A order to solve the Stokes-Biot model problem by efficient adaptive finite element methods, reliable and efficient a posteriori error analysis is important to provide appropriated indicators. In this section, we first define the local and global indicators (Section \ref{s3}) and then the lower and upper error bounds are derived (Sections (\ref{s4}) and (\ref{s5})).
\subsection{Residual error estimators} \label{s3}
The general philosophy of residual error estimators is to estimate an appropriate norm of the correct residual by terms that can be evaluated easier, and that involve the data at hand. To this end define the
exact element residuals:
\begin{definition} [\textbf{Exact Element Residuals}] Let $t\in]0,T]$ and\\ $\textbf{W}_h(t)=(\textbf{v}_{f,h}(t),p_{f,h}(t),\textbf{v}_{p,h}(t),
	w_{p,h}(t),\xi_{p,h}(t),\lambda_h(t))\in \textbf{V}_{f,h}\times W_{f,h}\times \textbf{V}_{p,h}\times W_{p,h}\times \textbf{X}_{p,h}\times \Lambda_h$ be an arbitrary finite element function. The exact element residuals over a triangle or tetrahedra $K\in\cT_h$ and over $E\in\cE_h(\Gamma_{fp})$ are defined for all $t\in]0,T]$ by:
\begin{eqnarray}
	\textbf{R}_{f,K}\left(\textbf{W}_h(t)\right)&=&\textbf{f}_{f}+\nabla\cdot \sigma_f(\textbf{v}_{f,h}(t),w_{f,h}(t))\\
	\textbf{R}_{p,K,1}(\textbf{W}_h(t))&=& \textbf{f}_{p}+\nabla\cdot\sigma_p(\xi_{p,h}(t),w_{p,h}(t))\\
	\textbf{R}_{p,K,2}(\textbf{W}_h(t))&=&\mu K^{-1}\textbf{v}_{p,h}(t)+\nabla w_{p,h}(t) 
	\\
	R_{f,K}(\textbf{W}_h(t))&=&q_f-\nabla\cdot \textbf{v}_{f,h}(t)\\
	R_{p,K}(\textbf{W}_h(t))&=& q_{p}-\partial_t\left(s_0 w_{p,h}(t)+\alpha\nabla\cdot \xi_{p,h}(t)\right)+\nabla\cdot \textbf{v}_{p,h}(t)\\
	R_{E,pf,1}(\textbf{W}_h(t))&=&\textbf{v}_{f,h}(t)\cdot\textbf{n}_{f,E}+
	\left(\partial_t\xi_p(t)+\textbf{v}_{p,h}(t)\right)\cdot\textbf{n}_{p,E}\\
	R_{E,pf,2}(\textbf{W}_h(t))&=&w_{p,h}+(\sigma_f(\textbf{v}_{f,h}(t),w_{f,h}(t))\textbf{n}_{f,E})\cdot\textbf{n}_{f,E}\\
	R_{E,pf,3}(\textbf{W}_h(t))&=& \sigma_f(\textbf{v}_{f,h}(t),w_{f,h}(t))\textbf{n}_{f,E}+
	\sigma_p(\xi_{p,h}(t),w_{p,h}(t))\textbf{n}_{p,E}\\
	R_{E,pf,4}(\textbf{W}_h(t))(j)&=&(\sigma_f(\textbf{v}_{f,h},w_{f,h}))\tau_{f,E,j}+\mu\alpha_{BJS}\sqrt{K_j^{-1}}\left(\textbf{v}_{f,h}-\partial_t\xi_{p,h}\right)\cdot\tau_{f,E,j}.
\end{eqnarray}
\end{definition}
As it is common, these exact residuals are replaced by some finite-dimensional approximation called approximate element residual $\textbf{r}_{\star,K}$, $r_{\star,K}$,  $\star\in\{f,p\}$, 
$R_{E,pf,l}$, $l\in\{1,2,3,4\}$.
This approximation is here achieved by projecting $\textbf{f}_f$ and $q_f$ on the space of piecewise constant functions in $\Omega_f$ and piecewise $\mathbb{P}^1$ functions in $\Omega_p$, more precisely for all $K\in\cT_h^f$ we take 
$$\textbf{f}_{f,K}=p_K(\textbf{f}_f) \mbox{  and }  q_{f,K}=P_K(q_f), \mbox{  with   } 
P_K: L^1(K)\rightarrow \mathbb{R} \mbox{  such  that } p_K(\phi)=\frac{1}{|K|}\int_K \phi(x)dx.$$ 
While for all $K\in\cT_h^p$, we take $\textbf{f}_{p,K}$ and $q_p$ as the unique element of $[\mathbb{P}^1(K)]^d$ respectively $\mathbb{P}^1(K)$ such that:
$$\int_K \textbf{f}_{p,K}(x)\cdot \textbf{q}(x)dx=\int_K \textbf{f}_{p}(x)\textbf{q}(x) dx \mbox{   }  \textbf{q}\in [\mathbb{P}^1(K)]^d$$
respectively,
$$\int_K q_{p,K}(x)q(x)dx=\int_K q_p(x)q(x)dx \forall  q\in \mathbb{P}^1(K).$$
Thereby, we define the approximate element residuals.
\begin{definition}
[\textbf{Approximate Element Residuals}]
 Let $t\in]0,T]$ and\\ $\textbf{W}_h(t)=(\textbf{v}_{f,h}(t),p_{f,h}(t),\textbf{v}_{p,h}(t),
w_{p,h}(t),\xi_{p,h}(t),\lambda_h(t))\in \textbf{V}_{f,h}\times W_{f,h}\times \textbf{V}_{p,h}\times W_{p,h}\times \textbf{X}_{p,h}\times \Lambda_h$ be an arbitrary finite element function.
Then, the approximate element residuals are defined
for all $t\in]0,T]$ by:
\begin{eqnarray}
\textbf{r}_{f,K}\left(\textbf{W}_h(t)\right)&=&\textbf{f}_{f,K}+\nabla\cdot \sigma_f(\textbf{v}_{f,h}(t),w_{f,h}(t))\\
\textbf{r}_{p,K,1}(\textbf{W}_h(t))&=& \textbf{f}_{p,K}+\nabla\cdot\sigma_p(\xi_{p,h}(t),w_{p,h}(t))\\
\textbf{r}_{p,K,2}(\textbf{W}_h(t))&=&\mu K^{-1}\textbf{v}_{p,h}(t)+\nabla w_{p,h}(t) 
\\
r_{f,K}(\textbf{W}_h(t))&=&q_{f,K}-\nabla\cdot \textbf{v}_{f,h}(t)\\
r_{p,K}(\textbf{W}_h(t))&=& q_{p,K}-\partial_t\left(s_0 w_{p,h}(t)+\alpha\nabla\cdot \xi_{p,h}(t)\right)+\nabla\cdot \textbf{v}_{p,h}(t)
\end{eqnarray}
\end{definition}
Next, introduce the gradient jump in normal direction by

\begin{eqnarray*}
	\left\{
	\begin{array}{cccccccccccccccccccccc}
	\textbf{J}_{E,\textbf{n}_E,f}(\textbf{U}_h)&:=&[(2\mu\textbf{D}(\textbf{u}_{f,h})-p_{f,h}\textbf{I})\cdot\textbf{n}_E]_E \mbox{   if  } E\in\cE_h(\Omega_f)\\
	\textbf{J}_{E,\textbf{n}_E,p}(\textbf{U}_h)&:=&[(2\mu\textbf{D}(\eta_{p,h})-p_{p,h}\textbf{I})\cdot\textbf{n}_E]_E \mbox{   if  } E\in\cE_h(\Omega_p).
	\end{array}
	\right.
\end{eqnarray*}
where $\textbf{I}$ is the identity matrix of $\mathbb{R}^{d\times d}$.
\begin{definition}\label{dindK}
	[\textbf{Residual Error Estimators}]
	Let $\textbf{U}_h=(\textbf{u}_{f,h}, p_{f,h}, \textbf{u}_{p,h},p_{p,h},\eta_{p,h},\lambda_h)$
	 be the finite
	 element solution of the problem (\ref{fh1})-(\ref{fh3}) 
	 in 
	 $L^{\infty}(0,T;\textbf{V}_{f,h})\times L^{\infty}(0,T;W_{f,h})\times L^{\infty}(0,T;\textbf{V}_{p,h})\times W^{1,\infty}(0,T;W_{p,h})\times 
	 W^{1,\infty}(0,T;\textbf{X}_{p,h})\times  L^{\infty}(0,T;\Lambda_h)$.
	 Then, the residual error estimator is locally defined by
	 \begin{eqnarray}\label{indK}
	 	\Theta_K(\textbf{U}_h):=\left[\Theta_{K,f}^2(\textbf{U}_h)+\Theta_
	 	{K,p}^2(\textbf{U}_h)+\Theta_{K,pf}^2(\textbf{U}_h)\right]^{\frac{1}{2}},
	 \end{eqnarray}
	 where 
	 \begin{eqnarray}\nonumber
	 	\Theta_{K,f}^2(\textbf{U}_h):=h_K^2\parallel \textbf{r}_{f,K}(\textbf{U}_h)\parallel_{L^{\infty}(0,T;L^2(K))}^2+\parallel r_{f,K}(\textbf{U}_h)\parallel_{L^{\infty}(0,T;L^2(K))}^2\\\label{indK1}
	 	+
	 	\displaystyle\sum_{E\in\cE_h(\Omega_f)}h_E\parallel\textbf{J}_{E,\textbf{n}_E,f}(\textbf{U}_h)\parallel_{L^{\infty}(0,T;L^2(E))}^2,
	 \end{eqnarray}
	 \begin{eqnarray}\nonumber
	 	\Theta_{K,p}^2(\textbf{U}_h):=
	 	h_K^2\left(\parallel \textbf{r}_{p,K,1}(\textbf{U}_h)\parallel_{L^{\infty}(0,T;L^2(K))}^2+
	 	\parallel \textbf{r}_{p,K,2}(\textbf{U}_h)\parallel_{L^{\infty}(0,T;L^2(K))}^2
	 	\right)
	 	\\\label{indK2}
	 	+h_K^2\parallel\curl [ \textbf{r}_{p,K,2}(\textbf{U}_h)] \parallel_{L^{\infty}(0,T;L^2(K))}^2
	 	+\parallel r_{p,K}(\textbf{U}_h) \parallel_{L^{\infty}(0,T;L^2(K))}^2\\\nonumber
	 	+\displaystyle\sum_{E\in\cE_h(\Omega_p)}h_E\parallel\textbf{J}_{E,\textbf{n}_E,p}(\textbf{U}_h)\parallel_{L^{\infty}(0,T;L^2(E))}^2
	 \end{eqnarray}
	 and 
	 \begin{eqnarray}\nonumber
	 	\Theta_{K,pf}^2(\textbf{U}_h):=
	 	\displaystyle\sum_{E\in\cE_h(\Gamma_{pf})}h_E\parallel R_{E,pf,1}(\textbf{U}_h)\parallel_{L^{\infty}
	 		(0,T;L^2(E))}^2\\\nonumber
 		+\displaystyle\sum_{E\in\cE_h(\Gamma_{pf})}h_E\parallel R_{E,pf,2}(\textbf{U}_h)\parallel_{L^{\infty}
 			(0,T;L^2(E))}^2\\\label{indK3}
 		+\displaystyle\sum_{E\in\cE_h(\Gamma_{pf})}h_E\parallel R_{E,pf,3}(\textbf{U}_h)\parallel_{L^{\infty}
 			(0,T;L^2(E))}^2\\\nonumber
 		+\displaystyle\sum_{E\in\cE_h(\Gamma_{pf})}h_E\left\{\displaystyle\sum_{j=1}^{d-1}
 		\parallel R_{E,pf,4}(j) \parallel_{L^{\infty}(0,T,L^2(E))}^2 \right\}.
	 \end{eqnarray}
	 The global residual error estimator is given by
	 \begin{eqnarray}\label{indKg}
	 	\Theta(\textbf{U}_h):=\left[\displaystyle\sum_{K\in\cT_h}\Theta_K(\textbf{U}_h)^2\right]^{\frac{1}{2}}.
	 \end{eqnarray}
 \end{definition}
Furthermore denote the local approximation terms by

\begin{eqnarray}
	\zeta_K^2:=\zeta_{K,f}^2+\zeta_{K,p}^2,
\end{eqnarray}
where
\begin{eqnarray*}
	\zeta_{K,f}^2:=
	h_K^2\parallel \textbf{R}_{f,K}(\textbf{U}_h)-\textbf{r}_{f,K}(\textbf{U}_h)\parallel_{L^{\infty}(0,T;L^2(K))}^2\\
	+h_K^2
	\parallel R_{f,K}(\textbf{U}_h)-r_{f,K}(\textbf{U}_h)\parallel_{L^{\infty}(0,T,L^2(K))}^2 \mbox{  if  } K\in\cT_h^f,
	\end{eqnarray*}
and
\begin{eqnarray*}
	\zeta_{K,p}^2:=
	h_K^2\parallel \textbf{R}_{p,K,1}(\textbf{U}_h)-\textbf{r}_{p,K,1}(\textbf{U}_h)\parallel_{L^{\infty}(0,T;L^2(K))}^2\\
	+h_K^2
	\parallel R_{p,K}(\textbf{U}_h)-r_{p,K}(\textbf{U}_h)\parallel_{L^{\infty}(0,T,L^2(K))}^2 \mbox{  if  } K\in\cT_h^p.
\end{eqnarray*}
The global approximation term is defined by
\begin{eqnarray}
	\zeta:=\left[\displaystyle\sum_{K\in\cT_h}\zeta_K^2\right]^{\frac{1}{2}}.
\end{eqnarray}
\begin{remark}
The residual character of each term on the right-hand sides of (\ref{indK})-(\ref{indK3}) is quite clear 
since if $\textbf{U}_h$ would be the exact solution of (\ref{stokes1})-{\ref{i3}}, then they would vanish.
\end{remark}
\subsection{Analytical tools}
\subsubsection{\textbf{Inverse inequalities}}
 In order to derive the lower error bounds, we proceed similarly as in \cite{carstensenandall} and 
\cite{Ca:97} (see also \cite{15}), by applying
inverse inequalities, and the localization technique based on simplex-bubble and face-bubble functions. To this end, we 
recall some notation and introduce further preliminary results. Given $K\in \mathcal{T}_h$, and 
$E\in \cE(K)$,
we let $b_K$ and $b_E$ be the usual simplexe-bubble and face-bubble 
functions respectively (see (1.5) and (1.6) in \cite{verfurth:96b}). In particular, $b_K$ satisfies 
$b_K\in \mathbb{P}^3(K)$, $supp(b_K)\subseteq K$, $b_K=0 \mbox{ on } \partial K$, and $0\leq b_K\leq 1\mbox{ on } K $.
Similarly, $b_E\in \mathbb{P}^2(K)$, $supp(b_E)\subseteq 
\omega_E:=\left\{K'\in \mathcal{T}_h:  E\in\cE (K')\right\}$, 
$b_E=0\mbox{  on  } \partial K\smallsetminus E$ and $0\leq b_E\leq 1\mbox{ in } \omega_E$.
We also recall from \cite{verfurth:94a} that, given $k\in\mathbb{N}$, there exists an extension operator
$L: C(E)\longrightarrow C(K)$ that satisfies $L(p)\in \mathbb{P}^k(K)$ and $L(p)_{|E}=p, \forall p\in \mathbb{P}^k(E)$.
A corresponding vectorial version of $L$, that is, the componentwise application of $L$, is denoted by 
$\textbf{L}$. Additional properties of $b_K$, $b_E$ and $L$ are collected in the following lemma (see \cite{verfurth:94a}):

\begin{lemma}
	Given $k\in \mathbb{N}^*$, there exist positive constants depending only on $k$ and shape-regularity of the triangulations 
	(minimum angle condition), such that for each simplexe $K$ and $E\in \cE(K)$ there hold
	\begin{eqnarray}\label{cl1}
	\parallel \phi \parallel_{K}&\lesssim&\parallel \phi b_K^{1/2}\parallel_{K}\lesssim
	\parallel \phi\parallel_{K}, \forall \phi\in \mathbb{P}^k(K)\\\label{cl2}
	|\phi b_K|_{1,K}&\lesssim&  h_K^{-1}\parallel \phi \parallel_{K}, \forall \phi\in \mathbb{P}^k(K)\\\label{cl3}
	\parallel \psi\parallel_{E}&\lesssim&\parallel b_E^{1/2}\psi\parallel_{E}\lesssim \parallel \psi\parallel_{E},
	\forall \psi\in \mathbb{P}^k(E)\\\label{cl4}
	\parallel L(\psi)\parallel_{K} +h_E|L(\psi)|_{1,K}&\lesssim& h_E^{1/2}\parallel \psi\parallel_{E}
	\forall \psi\in \mathbb{P}^k(E)
	\end{eqnarray}
\end{lemma}
\begin{lemma} \label{lemd1}
	(\textbf{Continuous trace inequality})
	There exists a positive constant  $\beta_1> 0$ depending only on  $\sigma_0$ such that 
	\begin{eqnarray}
	\parallel \textbf{v}\parallel_{\partial K}^2&\leqslant &\beta_1 \parallel\textbf{v}\parallel_{K}
	\parallel\textbf{v}\parallel_{1,K}, \mbox{  } \forall K\in \mathcal{T}_h, \forall \textbf{v}\in [H^1(K)]^d.
	\end{eqnarray}
\end{lemma}
\subsubsection{\textbf{Cl\'ement interpolation operator}}
In order to derive the upper error bounds, 
we introduce the  Cl\'ement interpolation operator 
$\mbox{I}_{\mbox{Cl}}^0: H_0^1(\Omega)\longrightarrow \mathcal{P}_c^b(\mathcal{T}_h)$ that approximates optimally non-smooth 
functions by continuous piecewise linear functions:
\begin{eqnarray*}
	\mathcal{P}_c^b(\mathcal{T}_h):=\left\{v\in C^0(\overline{\Omega}): \mbox{   }
	v_{|K} \in \mathbb{P}^1(K), \mbox{   } \forall K\in\mathcal{T}_h \mbox{  and  } 
	v=0 \mbox{  on  } \partial \Omega\right\}.
\end{eqnarray*}
In addition, we will make use of a vector valued version 
of $\mbox{I}_{\mbox{Cl}}^0$, that is,\\ $\textbf{I}_{\mbox{Cl}}^0: [H_0^1(\Omega)]^d\longrightarrow [\mathcal{P}_c^b(\mathcal{T}_h)]^d $, which 
is defined componentwise by $\mbox{I}_{\mbox{Cl}}^0.$ The following lemma establishes the local approximation properties of 
$\mbox{I}_{\mbox{Cl}}^0$ (and hence of $\textbf{I}_{\mbox{Cl}}^0$), for a proof see \cite[Section 3]{clement:75}.
\begin{lemma}\label{clement}
	There exist constants $C_1, C_2> 0$, independent of $h$, such that for all $v\in H_0^1(\Omega)$ there hold 
	\begin{eqnarray}
	\parallel v-\mbox{I}_{Cl}^0(v)\parallel_{K}&\leq& C_1h_K\parallel v\parallel_{1,\Delta(K)} 
	\mbox{   } \forall K\in \mathcal{T}_h,   \hspace*{0.2cm}\mbox{  and  }\\
	\parallel v-\mbox{I}_{Cl}^0(v)\parallel_{E}&\leq& C_2h_E^{1/2}\parallel v\parallel_{1,\Delta(E)}
	\mbox{   } \forall E\in\cE_h,
	\end{eqnarray}
	where $\Delta(K):=\cup\left\{K'\in \mathcal{T}_h: K'\cap K\neq\emptyset\right\}$ and  
	$\Delta (E):=\cup \left\{K'\in \mathcal{T}_h: K'\cap E\neq\emptyset\right\}$.
\end{lemma}
\subsubsection{\textbf{Helmholtz decomposition}}
\begin{lemma}(\cite{HJA:2017})\label{helmoltz}
	There exists $C_p> 0$ such that every $\textbf{v}_p\in \textbf{H}(\dive;\Omega_p)$ 
	can be decomposed as $\textbf{v}_p=\textbf{w}+\curl \beta$, where $\textbf{w}\in [H^1(\Omega_p)]^d$, 
	$\beta\in H^1(\Omega_p)$, $\int_{\Omega_p} \beta (x)dx=0$ and 
	\begin{eqnarray}\label{esth}
		\parallel \textbf{w}\parallel_{1,\Omega_p}+\parallel \beta \parallel_{1,\Omega_p}\leq C_p 
		\parallel \textbf{v}_p\parallel_{\textbf{H}(\dive;\Omega_p)}.
	\end{eqnarray}
\end{lemma}
\subsection{Reliability of the a posteriori error estimator}\label{s4}
We set  $\textbf{H}=L^{\infty}(0,T;\textbf{V}_f)\times L^{\infty}(0,T;W_f)\times L^{\infty}(0,T;\textbf{V}_p)\times W^{1,\infty}(0,T;W_p)\times W^{1,\infty}(0,T;\textbf{X}_p)\times
L^{\infty}(0,T;\Lambda)$
and 
$\textbf{H}_h=
L^{\infty}(0,T;\textbf{V}_{f,h})\times L^{\infty}(0,T;W_{f,h})\times L^{\infty}(0,T;\textbf{V}_{p,h})\times W^{1,\infty}(0,T;W_{p,h})\times W^{1,\infty}(0,T;\textbf{X}_{p,h})\times
L^{\infty}(0,T;\Lambda_h)$.

The first main result is given by the following theorem:
\begin{theorem} (\textbf{Upper Error Bound})
	Let $\textbf{U}=(\textbf{u}_f,p_f,\textbf{u}_p,p_p,\eta_p,\lambda)\in \textbf{H}$ be the exat solution and $\textbf{U}_h=(\textbf{u}_{f,h},p_{f,h},\textbf{u}_{p,h},p_{p,h},\eta_{p,h},\lambda_h)\in \textbf{H}_h$ be the finite element solution. Then, there exist a
	positive constant $C_{\mbox{rel}}$ such that
	the error is bounded globally from above by:
	\begin{eqnarray}
		\parallel \textbf{U}-\textbf{U}_h\parallel_{\textbf{H}_h}&\leq& C_{\mbox{rel}}\left[\Theta(\textbf{U}_h)+\zeta\right].
	\end{eqnarray}
\end{theorem}
\begin{proof}
	Let $\textbf{U}=(\textbf{u}_f,p_f,\textbf{u}_p,p_p,\eta_p,\lambda)\in\textbf{H}$ and 
	$\textbf{W}=(\textbf{v}_f,w_f,\textbf{v}_p,w_p,\xi_p,\mu)\in \textbf{H}$. For $t\in]0,T]$ we define the operator $\textbf{A}$ by
	\begin{eqnarray*}
		\textbf{A}(\textbf{U}(t),\textbf{W}(t))&:=& a_f(\textbf{u}_f(t),\textbf{v}_f(t))+a_p^d(\textbf{u}_p(t),\textbf{v}_p(t))+b_p(\textbf{v}_p(t),p_p(t))+\alpha b_p(\xi_p(t),p_p(t))\\
		&+&b_{\Gamma}(\textbf{v}_f(t),\textbf{v}_p(t),\xi_p(t),\lambda(t))+
		\left(s_0\partial_t p_p(t),w_p(t)\right)_{\Omega_p}-\alpha b_p(\partial_t \eta_p(t),w_p(t))\\
		&-& b_p(\textbf{u}_p(t),w_p(t))-b_f(\textbf{u}_f(t),w_f(t))+b_{\Gamma}(\textbf{u}_f(t),\textbf{u}_p(t),\partial_t \eta_p(t);\mu),
	\end{eqnarray*}
and \begin{eqnarray*}
	\textbf{F}(\textbf{W}(t)):=(\textbf{f}_f,\textbf{v}_f(t))_{\Omega_f}+(\textbf{f}_p,\xi_p(t))+(q_f,w_f(t))_{\Omega_f}+(q_p,w_p(t))_{\Omega_p}.
\end{eqnarray*}
Then the continuous problem (\ref{f1})-(\ref{f3}) is equivalent to: Find $\textbf{U}\in \textbf{H}$ such that for all $t\in ]0,T]$, we have:
\begin{eqnarray}\label{fcm}
\textbf{A}(\textbf{U}(t),\textbf{W}(t))=\textbf{F}(\textbf{W}(t)), \mbox{  } \forall \textbf{W}\in \textbf{H}.
\end{eqnarray}
We define the discrete version by the same way: Find $\textbf{U}_h\in\textbf{H}_h$ such that for all $t\in ]0,T]$,
\begin{eqnarray}\label{fdm}
\textbf{A}(\textbf{U}_h(t),\textbf{W}_h(t))=\textbf{F}(\textbf{W}_h(t)), \mbox{  } \forall \textbf{W}_h\in \textbf{H}_h.
\end{eqnarray}
Since for all $t\in]0,T], $ and $\textbf{W}_h\in\textbf{H}_h$
$\textbf{A}(\textbf{U}(t)-\textbf{U}_h(t),\textbf{W}_h(t))=0$, then from (\ref{fcm}) we obtain
\begin{eqnarray*}
	\textbf{A}(\textbf{U}(t)-\textbf{U}_h(t),\textbf{W}(t))&=& \textbf{A}(\textbf{U}(t)-\textbf{U}_h(t), \textbf{W}(t)-\textbf{W}_h(t))\\
	&=& \textbf{A}(\textbf{U}(t), \textbf{W}(t)-\textbf{W}_h(t))-
	\textbf{A}(\textbf{U}_h(t), \textbf{W}(t)-\textbf{W}_h(t))\\
	&=&\textbf{F}(\textbf{W}(t)-\textbf{W}_h(t))-\textbf{A}(\textbf{U}_h(t), \textbf{W}(t)-\textbf{W}_h(t))
	\\
	&=&
	(\textbf{f}_f,\textbf{v}_f(t)-\textbf{v}_{f,h}(t))_{\Omega_f}+(\textbf{f}_p,\xi_p(t)-\xi_{p,h}(t))_{\Omega_p}\\
	&+&(q_f,w_f(t)-w_{f,h}(t))_
	{\Omega_f}+(q_p,w_p(t)-w_{p,h}(t))_{\Omega_p}\\
	&-&\textbf{A}(\textbf{U}_h(t), \textbf{W}(t)-\textbf{W}_h(t))\\
	&=&(\textbf{f}_f-\textbf{f}_{f,h},\textbf{v}_f(t)-\textbf{v}_{f,h}(t))_{\Omega_f}+
	(\textbf{f}_p-\textbf{f}_{p,h}
	,\xi_p(t)-\xi_{p,h}(t))_{\Omega_p}\\
	&+&(q_f-q_{f,h},w_f(t)-w_{f,h}(t))_
	{\Omega_f}
	+(q_p-q_{p,h},w_p(t)-w_{p,h}(t))_{\Omega_p}\\
	&+&(\textbf{f}_{f,h},\textbf{v}_f(t)-\textbf{v}_{f,h}(t))_{\Omega_f}+
	(\textbf{f}_{p,h}
	,\xi_p(t)-\xi_{p,h}(t))_{\Omega_p}\\
	&+&(q_{f,h},w_f(t)-w_{f,h}(t))_
	{\Omega_f}
	+(q_{p,h},w_p(t)-w_{p,h}(t))_{\Omega_p}\\
	&-&\textbf{A}(\textbf{U}_h(t), \textbf{W}(t)-\textbf{W}_h(t)).
\end{eqnarray*}
Hence,
\begin{eqnarray}\label{identity}
		\textbf{A}(\textbf{U}(t)-\textbf{U}_h(t),\textbf{W}(t))=\displaystyle\sum_{K\in\cT_h}
			\textbf{A}_K(\textbf{U}(t)-\textbf{U}_h(t),\textbf{W}(t)),
\end{eqnarray}
where,
\begin{eqnarray*}
	\textbf{A}_K(\textbf{U}(t)-\textbf{U}_h(t),\textbf{W}(t))&:=&
	(\textbf{f}_f-\textbf{f}_{f,h},\textbf{v}_f(t)-\textbf{v}_{f,h}(t))_{\Omega_f\cap K}+
	(\textbf{f}_p-\textbf{f}_{p,h}
	,\xi_p(t)-\xi_{p,h}(t))_{\Omega_p\cap K}\\
	&+&(q_f-q_{f,h},w_f(t)-w_{f,h}(t))_
	{\Omega_f\cap K}
	+(q_p-q_{p,h},w_p(t)-w_{p,h}(t))_{\Omega_p\cap K}\\
	&+&(\textbf{f}_{f,h},\textbf{v}_f(t)-\textbf{v}_{f,h}(t))_{\Omega_f\cap K}+
	(\textbf{f}_{p,h}
	,\xi_p(t)-\xi_{p,h}(t))_{\Omega_p\cap K}\\
	&+&(q_{f,h},w_f(t)-w_{f,h}(t))_
	{\Omega_f\cap K}
	+(q_{p,h},w_p(t)-w_{p,h}(t))_{\Omega_p\cap K}\\
	&-&\textbf{A}_K(\textbf{U}_h(t), \textbf{W}(t)-\textbf{W}_h(t))\\
	&=& (\textbf{R}_{f,K}(\textbf{U}_h(t))-\textbf{r}_{f,K}(\textbf{U}_h(t)),\textbf{v}_f(t)-
	\textbf{v}_{f,h}(t))_{\Omega_f\cap K}\\
	&+&
	(R_{f,K}(\textbf{U}_h(t))-r_{f,K}(\textbf{U}_h(t)),w_f(t)-w_{f,h}(t))_{\Omega_f\cap K}\\
	&+&
	(\textbf{R}_{p,K,1}(\textbf{U}_h(t))-\textbf{r}_{p,K,1}(\textbf{U}_h(t)),\xi_p(t)-\xi_{p,h}(t))_{\Omega_p\cap K}\\
	&+&
	(R_{p,K}(\textbf{U}_h(t))-r_{p,K}(\textbf{U}_h(t)),w_p(t)-w_{p,h}(t))_{\Omega_p\cap K}\\
	&+&\textbf{B}_K(\textbf{U}_h(t),\textbf{W}(t)-\textbf{W}_h(t)),
\end{eqnarray*}
with,
\begin{eqnarray*}
	\textbf{B}_K(\textbf{U}_h(t),\textbf{W}(t)-\textbf{W}_h(t))&=&(\textbf{f}_{f,K},\textbf{v}_f(t)-\textbf{v}_{f,h}(t))_{\Omega_f\cap K}+
	(\textbf{f}_{p,K},\xi_p(t)-\xi_{p,h}(t))_{\Omega_p\cap K}\\
	&+&(q_{f,K},w_{f}(t)-w_{f,h}(t))_{\Omega_f\cap K}
	+
	(q_{p,K},w_{p}(t)-w_{p,h}(t))_{\Omega_p\cap K}\\
	&-&\textbf{A}_K(\textbf{U}_h(t),\textbf{W}(t)-\textbf{W}_h(t)).
\end{eqnarray*}
$\textbf{W}(t)=(\textbf{v}_f(t),w_f(t),\textbf{v}_p(t),w_p(t),\xi_p(t),\mu(t))$ and we take 
$\textbf{W}_h(t)=(\textbf{v}_{f,h}(t),0,\textbf{v}_{p,h}(t),0,\xi_{p,h}(t),0)$ with
$\textbf{v}_{f,h}(t)=\textbf{I}_{\mbox{Cl}}^0(\textbf{v}_{f}(t))$ and $\xi_{p,h}(t)=\textbf{I}_{\mbox{Cl}}^0(\xi_p(t))$. As $\textbf{v}_p(t)\in \textbf{H}(\dive;\Omega_p)$, then by Theorem \ref{helmoltz}, $\textbf{v}_p(t)$ admits the decomposition 
$\textbf{v}_p(t)=\textbf{w}_p(t)+\curl \beta(t)$ where  $\textbf{w}_p(t)\in [H^1(\Omega_p)]^d$ and $\beta(t)\in H^1( \Omega_p)$ with $\int_{\Omega_p}\beta(t)(x) dx=0$ and 
$\parallel\textbf{w}_p\parallel_{L^{\infty}(0,T;[H^1(\Omega_p)]^d)}+
\parallel\beta\parallel_{L^{\infty}(0,T;H^1(\Omega_p))}\leq C_p \parallel\textbf{v}_p\parallel_{L^{\infty}(0,T; \textbf{V}_p)}$. We consider 
$\textbf{v}_{p,h}(t)=\textbf{w}_{p,h}(t)+\curl \beta_{p,h}(t)$ with $\textbf{w}_{p,h}(t)=\textbf{I}_{\mbox{Cl}}(\textbf{w}_p(t))$ and $\beta_{p,h}(t)=I_{\mbox{Cl}}(\beta_p(t))$. Thus $\textbf{v}_p(t)-\textbf{v}_{p,h}(t)=
(\textbf{w}_p(t)-\textbf{w}_{p,h}(t))+\curl (\beta_p(t)-\beta_{p,h}(t))$.
Therefore, integrate by parts element by element we may write:
\begin{eqnarray*}
	\textbf{B}_K(\textbf{U}_h(t),\textbf{W}(t)-\textbf{W}_h(t))&=&
	(\textbf{r}_{f,K}(\textbf{U}_h(t)),\textbf{v}_f(t)-\textbf{v}_{f,h}(t))_{\Omega_f\cap K}\\
	&+&(\textbf{r}_{p,K,1}(\textbf{U}_h(t)),\xi_p(t)-\xi_{p,h}(t))_{\Omega_p\cap K}\\
	&-& (\textbf{r}_{p,K,2}(\textbf{U}_h(t)),\textbf{w}_p(t)-\textbf{w}_{p,h}(t))_{\Omega_p\cap K}\\
	&-&(\curl \textbf{r}_{p,K,2}(\textbf{U}_h(t)),\beta_p(t)-\beta_{p,h}(t) )_{\Omega_p\cap K}\\
	&+& (r_{f,K}(\textbf{U}_h(t)),w_f)_{\Omega_f\cap K}+(r_{p,K}(\textbf{U}_h(t)),w_p)_{\Omega_p\cap K}\\
	&-&\displaystyle\sum_{E\in\cE_h(\Omega_f\cap K)} (\textbf{J}_{E,\textbf{n}_E,f}(\textbf{U}_h(t)),\textbf{v}_f(t)-\textbf{v}_{f,h}(t))_E\\
	&-&
	\displaystyle\sum_{E\in\cE_h(\Omega_p\cap K)} (\textbf{J}_{E,\textbf{n}_E,p}(\textbf{U}_h(t)),\xi_p(t)-\xi_{p,h}(t))_E\\
	&+&
	\displaystyle\sum_{E\in\cE_h(\Gamma_{fp})}
	(R_{E,pf,1}(\textbf{U}_h(t)),(\textbf{v}_f(t)-\textbf{v}_{f,h}(t))\cdot\textbf{n}_{f,E})_E\\
	&+&
	\displaystyle\sum_{E\in\cE_h(\Gamma_{fp})}
	(R_{E,pf,2}(\textbf{U}_h(t)),(\textbf{v}_f(t)-\textbf{v}_{f,h}(t))\cdot\textbf{n}_{f,E})_E\\
	&+&
	\displaystyle\sum_{E\in\cE_h(\Gamma_{fp})}
	(R_{E,pf,3}(\textbf{U}_h(t)),(\xi_p(t)-\xi_{p,h}(t)))_E\\
	&-&
	\displaystyle\sum_{E\in\cE_h(\Gamma_{fp})}
	\displaystyle\sum_{j=1}^{d-1}(R_{E,pf,4}(\textbf{U}_h(t))(j),\textbf{v}_f(t)-\textbf{v}_{f,h}(t)\cdot\tau_j)_{E}
	\end{eqnarray*}
Coercivity of operator $\textbf{A}$ leads to inf-sup condition:
\begin{eqnarray}\label{inf-sup}
	\parallel \textbf{U}-\textbf{U}_h\parallel_{\textbf{H}_h}\leq 
	\displaystyle\sup_{\textbf{W}\in \textbf{H}}\frac{|\textbf{A}(\textbf{U}-\textbf{U}_h,\textbf{W})|}
	{\parallel \textbf{W}\parallel_{\textbf{H}}}.
\end{eqnarray}
By consequently , the identity (\ref{identity}), inf-sup condition of operator $\textbf{A}$ (\ref{inf-sup}), Cauchy-Schwarz inequality, estimation of Lemma (\ref{helmoltz}) and the approximation properties of Lemma \ref{clement} imply the required estimate and finish the proof.
\end{proof}
\subsection{Efficiency of the a posteriori error estimator}\label{s5}
For $K\in\cT_h$, we set 
\begin{eqnarray*}
\textbf{H}_h(K)&=&
L^{\infty}(0,T;\textbf{V}_{f,h}(K))\times L^{\infty}(0,T;W_{f,h}(K))\times L^{\infty}(0,T;\textbf{V}_{p,h}(K))\\
&\times& W^{1,\infty}(0,T;W_{p,h}(K))
\times W^{1,\infty}(0,T;\textbf{X}_{p,h}(K))\times
L^{\infty}(0,T;\Lambda_h(K)),
\end{eqnarray*}
where
\begin{eqnarray*}
	\textbf{V}_{f,h}(K):=\left\{{\textbf{v}_{f,h}}_{|K}: \textbf{v}_{f,h}\in\textbf{V}_{f,h} \right\};
	\mbox{     }
	W_{f,h}(K):=\left\{ {w_{f,h}}_{|K}: w_{f,h}\in W_{f,h}\right\};\\
	\textbf{V}_{p,h}(K):= \left\{{\textbf{v}_{p,h}}_{|K}: \textbf{v}_{p,h}\in\textbf{V}_{p,h} \right\};
	\mbox{      }	W_{p,h}(K):=\left\{ {w_{p,h}}_{|K}: w_{p,h}\in W_{p,h}\right\};\\
	\textbf{X}_{p,h}(K):=\left\{{\xi_{p,h}}_{|K}: \xi_{p,h}\in\textbf{X}_{p,h}\right\};
	\mbox{  and  }
	\Lambda_h(K):=\left\{{\lambda_h}_{|K}: \lambda_h\in \Lambda_h\right\}.
\end{eqnarray*}
The error estimator $\Theta(\textbf{U}_h)$ is consider efficient if it satisfies the following theorem:
\begin{theorem}
	[\textbf{Lower Error Bound}]
	Let $\textbf{U}=(\textbf{u}_f,p_f,\textbf{u}_p,p_p,\eta_p,\lambda)\in \textbf{H}$ be the exat solution and $\textbf{U}_h=(\textbf{u}_{f,h},p_{f,h},\textbf{u}_{p,h},p_{p,h},\eta_{p,h},\lambda_h)\in \textbf{H}_h$ be the finite element solution. Then, there exist a
	positive constant $C_{\mbox{eff}}$  such that
	the error is bounded locally from below for all $K\in\cT_h$ by:
	\begin{eqnarray}
		\Theta_K(\textbf{U}_h)\leq C_{\mbox{eff}} \left[\parallel\textbf{U}-\textbf{U}_h\parallel_{\textbf{H}_h(K)}+\displaystyle\sum_{K'\subset \tilde{w}_K} \zeta_K\right],
	\end{eqnarray}
	where $\tilde{w}_K$ is a finite union of neighborning elements of $K$.
	 
\end{theorem}
\begin{proof}
	We begin by bounding each the residuals separately.
\end{proof}
\section{Discussion}\label{s6}
In this paper we have discussed a posteriori error estimates for a finite element approximation of the Stokes-Biot system.  A residual type a posteriori error estimator is provided, that is both reliable and efficient. Many issues remain to be addressed in this area, let us mention other types of a posteriori error estimators and nonconforming finite element methods or implementation and convergence analysis of adaptive finite element methods. Further it is well known that an internal layer appears at the interface $\Gamma_{fp}$ as the permeability tensor degenerates, in that case anisotropic meshes have to be used in this layer (see for instance \cite{H:2015}). Hence we intend to extend our results to such anisotropic meshes.
\section{Nomenclatures}
\begin{itemize}
	\item $\Omega\subset \mathbb{R}^d, d\in\{2,3\}$ bounded domain 
	\item $\Omega_p:$ the poroelastic medium domain
	\item $\Omega_f=\Omega\smallsetminus \overline{\Omega}_d$
	\item $\Gamma_{fp}=\partial \Omega_f\cap\partial \Omega_p$
	\item $\Gamma_{\star}=\partial \Omega_{\star}\smallsetminus\Gamma_{fp},$ $\star=f,p$
	\item $\textbf{n}_f$ (resp. $\textbf{n}_p$) the unit outward normal vector along $\partial \Omega_f$ (resp. $\partial \Omega_p$)
	\item $\textbf{u}_f$: the fluid velocity in $\Omega_f$
	\item $p_f$: the fluid pressure in $\Omega_f$
	\item $\textbf{u}_p, \eta_p$: the fluid velocities in $\Omega_p$
	\item $p_p$: the fluid pressure in $\Omega_p$
	
	\item 
	In $2D$, the $\curl$ of a scalar function $w$ is given as usual by 
	$$\curl w:=\left(\frac{\partial w}{\partial x_2},-\frac{\partial w}{\partial x_1}\right)^{\top}$$
	\item 
	In $3D$, the $\curl $ of a vector function $\textbf{w}=(w_1,w_2,w_3)$
	is given as usual by $\curl \textbf{w}:=\nabla \times \textbf{w}$ namely,
	\begin{eqnarray*}
		\curl \textbf{w}&:=& 
		\left(\frac{\partial w_3}{\partial x_2}-\frac{\partial w_2}{\partial x_3},
		\frac{\partial w_1}{\partial x_3}-\frac{\partial w_3}{\partial x_1},
		\frac{\partial w_2}{\partial x_1}-\frac{\partial w_1}{\partial x_2}\right)
	\end{eqnarray*}
	\item $\mathbb{P}^k$: the space of polynomials of total degree not larger than $k$
	\item $\cT_h$: triangulation of $\Omega$
	\item $\cT_h^{\star}$: the corresponding induced triangulation of $\Omega_{\star}$, $\star\in\{f,p\}$
	\item For any $K\in\cT_h$, $h_K$ is the diameter of $K$ and $\rho_K=2r_K$ is the diameter of the largest ball inscribed into $K$
	\item $h:=\displaystyle\max_{K\in\cT_h} h_K$  and   $\sigma_h:=\displaystyle\max_{K\in\cT_h} \frac{h_K}{\rho_K}$
	\item $\cE_h$: the set of all the edges or faces of the triangulation
	\item $\cE(K)$: the set of all the edges ($N=2$) or faces ($N=3$) of a element $K$
	\item $\cE_h:=\displaystyle\bigcup_{K\in\cT_h} \cE(K)$
	\item $\mathcal{N}(K)$: the set of all the vertices of a element $K$
	\item $\mathcal{N}_h:=\displaystyle\bigcup_{K\in\cT_h} \mathcal{N}(K)$
	\item For $\mathcal{A}\subset \overline{\Omega}$, $\mathcal{E}_h(\mathcal{A}):=\{E\in\mathcal{E}_h: E\subset \mathcal{A}\}$
	\item For $E\in\cE_h(\Omega_{\star})$, we associate a unit vector $\textbf{n}_{E,\star}$ such that $\textbf{n}_{E,\star}$ is orthogonal to $E$ and equals to the unit exterior normal vector to $\partial\Omega_{\star}$, $\star\in\{f,p\}$
	\item For $E\in\cE_h(\Omega_{\star})$, $[\phi]_E$ is the jump across $E$ in the direction of $\textbf{n}_{E,\star}$
	\item In order to avoid excessive use of constants, the abbreviations $x\lesssim y$ and $x\sim y$ 
	stand for $x\leqslant cy$ and $c_1x\leqslant y \leqslant c_2x$, respectively, with positive constants independent of $x$, $y$ or $\cT_h$
	\item $\partial_t=\frac{\partial}{\partial_t}$
	\item $\parallel \phi\parallel_{L^2(0,T;X)}:=\left(\int_{0}^{T}\parallel\phi(s)\parallel_X ds\right)^{1/2}$
	\item 
	$\parallel\phi\parallel_{L^{\infty}(0,T;X)}:=\displaystyle\sup_{0< t\leq T}\parallel\phi(t)\parallel_{X}$
	\item 
	$\parallel\phi\parallel_{W^{1,\infty}(0,T;X)}:=\displaystyle\sup_{0< t\leq T}\left\{
	\parallel\phi(t)\parallel_X, \parallel\partial_t\phi\parallel_X\right\}$.
\end{itemize}

\end{document}